\documentclass[reqno]{amsart}

\usepackage{amssymb}
\usepackage{graphicx}
\usepackage{amscd}
\usepackage[hidelinks]{hyperref}
\usepackage{color}
\usepackage{float}
\usepackage{graphics,amsmath,amssymb}
\usepackage{amsthm}
\usepackage{amsfonts}
\usepackage{latexsym}
\usepackage{epsf}
\usepackage{xifthen}
\usepackage{mathrsfs}
\usepackage{dsfont}
\usepackage{makecell}
\usepackage[FIGTOPCAP]{subfigure}
\usepackage{amsmath}
\allowdisplaybreaks[4]
\usepackage{listings}
\usepackage{etoolbox}
\usepackage{fancyhdr}
\usepackage{pdflscape}
\usepackage[title,toc,titletoc]{appendix}
\usepackage{enumitem}
\usepackage[noadjust]{cite}
\usepackage{forest}

 \newtheoremstyle{mytheorem}
 {3pt}
 {3pt}
 {\slshape}
 {}
 {\bfseries}
 {.}
 { }
 {}

\numberwithin{equation}{section}

\theoremstyle{theorem}
\newtheorem{theorem}{Theorem}[section]
\newtheorem*{theorem*}{Theorem}

\newtheorem{lemma}[theorem]{Lemma}

\newtheorem{innercustomgeneric}{\customgenericname}
\providecommand{\customgenericname}{}
\newcommand{\newcustomtheorem}[2]{%
	\newenvironment{#1}[1]
	{%
		\renewcommand\customgenericname{#2}%
		\renewcommand\theinnercustomgeneric{##1}%
		\innercustomgeneric
	}
	{\endinnercustomgeneric}
}
\newcustomtheorem{ctheorem}{Theorem}

\theoremstyle{definition}

\newtheorem*{example*}{Example}
\newtheorem{conjecture}{Conjecture}[section]

\theoremstyle{remark}

\newtheorem*{remark*}{Remark}
\newtheorem*{remarks*}{Remarks}

\newtheoremstyle{named}{}{}{\itshape}{}{\bfseries}{.}{.5em}{#1\thmnote{ #3}}
\theoremstyle{named}

\newcommand{\Keywords}[1]{\ifthenelse{\isempty{#1}}{}{\smallskip \smallskip \noindent \textbf{Keywords}. #1}}
\newcommand{\MSC}[2][2020]{\ifthenelse{\isempty{#2}}{}{\smallskip \smallskip \noindent \textbf{#1MSC}. #2}}
\newcommand{\abstractnote}[1]{\ifthenelse{\isempty{#1}}{}{\smallskip \smallskip \noindent \textsuperscript{\dag}#1}}

\makeatletter
\def\specialsection{\@startsection{section}{1}%
  \z@{\linespacing\@plus\linespacing}{.5\linespacing}%
  {\normalfont}}
\def\section{\@startsection{section}{1}%
  \z@{.7\linespacing\@plus\linespacing}{.5\linespacing}%
  {\normalfont\scshape}}
\patchcmd{\@settitle}{\uppercasenonmath\@title}{\Large\boldmath}{}{}
\patchcmd{\@settitle}{\begin{center}}{\begin{flushleft}}{}{}
\patchcmd{\@settitle}{\end{center}}{\end{flushleft}}{}{}
\patchcmd{\@setauthors}{\MakeUppercase}{\normalsize}{}{}
\patchcmd{\@setauthors}{\centering}{\raggedright}{}{}
\patchcmd{\section}{\scshape}{\large\bfseries\boldmath}{}{}
\patchcmd{\subsection}{\bfseries}{\bfseries\boldmath}{}{}
\renewcommand{\@secnumfont}{\bfseries}
\patchcmd{\@startsection}{\@afterindenttrue}{\@afterindentfalse}{}{}
\patchcmd{\abstract}{\leftmargin3pc}{\leftmargin1pc}{}{}

\def\maketitle{\par
  \@topnum\z@ 
  \@setcopyright
  \thispagestyle{empty}
  \ifx\@empty\shortauthors \let\shortauthors\shorttitle
  \else \andify\shortauthors
  \fi
  \@maketitle@hook
  \begingroup
  \@maketitle
  \toks@\@xp{\shortauthors}\@temptokena\@xp{\shorttitle}%
  \toks4{\def\\{ \ignorespaces}}
  \edef\@tempa{%
    \@nx\markboth{\the\toks4
      \@nx\MakeUppercase{\the\toks@}}{\the\@temptokena}}%
  \@tempa
  \endgroup
  \c@footnote\z@
  \@cleartopmattertags
}
\makeatother

\newcommand{\eee}{\mathrm{e}}
\newcommand{\iii}{\mathrm{i}}
\newcommand{\ddd}{\operatorname{d}}

\title[Asymmetric Rogers--Ramanujan type identities. I]{Asymmetric Rogers--Ramanujan type identities. I.\\The Andrews--Uncu Conjecture}

\author[S. Chern]{Shane Chern}
\address{Department of Mathematics and Statistics, Dalhousie University, Halifax, Nova Scotia, B3H 4R2, Canada}
\email{chenxiaohang92@gmail.com}

\date{}

\begin{document}

\maketitle

\begin{abstract}

In this work, we start an investigation of asymmetric Rogers--Ramanujan type identities. The first object is the following unexpected relation
$$\sum_{n\ge 0} \frac{(-1)^n q^{3\binom{n}{2}+4n}(q;q^3)_n}{(q^9;q^9)_n} = \frac{(q^{4};q^{6})_\infty (q^{12};q^{18})_\infty}{(q^{5};q^{6})_\infty (q^{9};q^{18})_\infty}$$
and its $a$-generalization. We then use this identity as a key ingredient to confirm a recent conjecture of G. E. Andrews and A. K. Uncu.

\Keywords{Rogers--Ramanujan type identity, Andrews--Uncu Conjecture, asymmetry, $a$-generalization.}

\MSC{11P84, 33D60.}
\end{abstract}

\section{Introduction}

The Rogers--Ramanujan identities,
\begin{align}
\sum_{n\ge 0}\frac{q^{n^2}}{(1-q)(1-q^2)\cdots (1-q^n)} &= \prod_{n\ge 0}\frac{1}{(1-q^{5n+1})(1-q^{5n+4})},\label{eq:RR-1}\\
\sum_{n\ge 0}\frac{q^{n^2+n}}{(1-q)(1-q^2)\cdots (1-q^n)} &= \prod_{n\ge 0}\frac{1}{(1-q^{5n+2})(1-q^{5n+3})},\label{eq:RR-2}
\end{align}
first appeared in L. J. Rogers' \textit{Second memoir on the expansion of certain infinite products}, which was completely overlooked for over two decades. They bloomed again after S. Ramanujan's rediscovery (without proof) \cite{Ram1914}, and then Ramanujan and Rogers \cite{RR1919} presented a joint new proof. Independently, I. Schur \cite{Sch1917} also published two fundamentally different proofs. The proofs of the Rogers--Ramanujan identities were later perfected by W. N. Bailey \cite{Bai1947,Bai1948} through the ingenious Bailey's Lemma, which finally led to L. J. Slater's list \cite{Sla1951,Sla1952} of 130 identities of Rogers--Ramanujan type.

As usual, we adopt the $q$-Pochhammer symbol for $n\in\mathbb{N}\cup\{\infty\}$,
\begin{align*}
(A;q)_n&:=\prod_{k=0}^{n-1} (1-A q^k),\\
(A_1,A_2,\ldots,A_m;q)_n&:=(A_1;q)_n(A_2;q)_n\cdots (A_m;q)_n.
\end{align*}
Throughout, we always assume that $q$ is a complex variable such that $|q|<1$. We shall also introduce other complex parameters and they are such that zero denominators are avoided.

In light to \eqref{eq:RR-1} and \eqref{eq:RR-2}, a Rogers--Ramanujan type identity asserts the equality of a $q$-summation and an infinite product. Also, on the summation side, we often require that the summand contains a factor in which the power of $q$ is quadratic in the summation variable(s).

Further, we may distinguish Rogers--Ramanujan type identities according to the shape of the summation and product sides:

For the summation side, we mainly focus on whether it is a \textit{single} or \textit{multiple} series.

For the product side, we say that it is \textit{periodic} if there exists a positive integer $M$ such that the infinite product has the form
\begin{align*}
\prod_{m=1}^M (q^m;q^M)_\infty^{\delta_m}.
\end{align*}
Further, for a periodic Rogers--Ramanujan type identity, we say that it is \textit{symmetric} if the powers $\delta_m$ satisfy $\delta_m=\delta_{M-m}$ for all $1\le m\le M-1$, and \textit{asymmetric} otherwise.

Most known Rogers--Ramanujan type identities are symmetric. In particular, all 130 identities in Slater's list are single-symmetric, and the infinite products in them may be deduced through Jacobi's triple product identity or the quintuple product identity. There are also many multiple-symmetric identities of Rogers--Ramanujan type, among which the most important one is the generalization of \eqref{eq:RR-1} and \eqref{eq:RR-2} due to G. E. Andrews \cite{And1974} and B. Gordon \cite{Gor1961}:
\begin{equation*}
\prod_{\substack{n\ge 1\\n\not\equiv 0,\pm i\, (\operatorname{mod}\, 2k+1)}}\frac{1}{1-q^n}=\sum_{n_1,\ldots,n_{k-1}\ge 0}\frac{q^{N_1^2+N_2^2+\cdots N_{k-1}^2+N_i+N_{i+1}+\cdots +N_{k-1}}}{(q;q)_{n_1}(q;q)_{n_2}\cdots (q;q)_{n_{k-1}}},
\end{equation*}
where $N_j=n_j+n_{j+1}+\cdots +n_{k-1}$.

In contrast, asymmetric Rogers--Ramanujan type identities are relatively rare. The first such instances are the little G\"ollnitz identities discovered by H. G\"ollnitz \cite{Gol1967}. They have the analytic form:
\begin{align*}
\sum_{n\ge 0}\frac{q^{n^2+n}(-q^{-1};q^2)_n}{(q^2;q^2)_n} &= \frac{1}{(q,q^5,q^6;q^8)_\infty},\\
\sum_{n\ge 0}\frac{q^{n^2+n}(-q;q^2)_n}{(q^2;q^2)_n} &= \frac{1}{(q^2,q^3,q^7;q^8)_\infty}.
\end{align*}
In fact, they are special cases of an identity due to V. A. Lebesgue \cite{Leb1840}:
\begin{align}\label{eq:Leb}
\sum_{n\ge 0}\frac{q^{\binom{n}{2}+n}(a;q)_n}{(q;q)_n} = \frac{(aq;q^2)_\infty}{(q;q^2)_\infty}.
\end{align}
Recently, S. Corteel, C. D. Savage and A. V. Sills \cite{CSS2012} also discovered two single-asymmetric Rogers--Ramanujan type identities:
\begin{align*}
\sum_{n\ge 0}\frac{q^{3\binom{n}{2}+n}(q^2;q^6)_n}{(q;q)_{3n}}&=\frac{1}{(q;q^3)_\infty (q^5;q^6)_\infty},\\
\sum_{n\ge 0}\frac{q^{3\binom{n}{2}+2n}(q^4;q^6)_n}{(q;q)_{3n+1}}&=\frac{1}{(q^2;q^3)_\infty (q;q^6)_\infty},
\end{align*}
which follow from a specialization of Heine's $q$-Gau\ss{} summation:
\begin{align}\label{eq:CSS}
\sum_{n\ge 0}\frac{(-1)^n a^n q^{\binom{n}{2}}(a;q)_n}{(a^2;q)_n (q;q)_n} = \frac{(a;q)_\infty}{(a^2;q)_\infty}.
\end{align}
See \cite[p.~62]{CSS2012}. For multiple-asymmetric Rogers--Ramanujan type identities, an important source is some ``Mod $12$'' Conjectures due to S. Kanade and M. C. Russell \cite{KR2015} and M. C. Russell \cite{Rus2016}. Their original conjectures are partition-theoretic and the corresponding analytic identities were later established by K. Kur\c{s}ung\"{o}z \cite{Kur2018}, Kanade and Russell themselves \cite{KR2018}, and S. Chern and Z. Li \cite{CL2020}. Finally, these conjectures were proved by K. Bringmann, C. Jennings-Shaffer and K. Mahlburg \cite{BJM2020} and H. Rosengren \cite{Ros2021}. For example, the analytic form of \cite[Conjecture $I_5$]{KR2015} is (see \cite[(1.15)]{BJM2020})
\begin{align*}
&\sum_{n_1,n_2,n_3\ge 0}\frac{q^{\binom{n_1}{2}+6\binom{n_2}{2}+9\binom{n_3}{2}+2 n_1 n_2+6n_2 n_3+3n_3n_1+n_1+4n_2+7n_3}}{(q;q)_{n_1} (q^2;q^2)_{n_2}  (q^3;q^3)_{n_3}}\\
&\qquad\qquad\qquad\qquad\qquad\qquad\qquad\qquad\qquad\qquad=\frac{1}{ (q;q^3)_\infty (q^3, q^6, q^{11} ; q^{12})_\infty }.
\end{align*}
Recently, G. E. Andrews and A. K. Uncu \cite[Theorem 1.1]{AU2021} proved one more multiple-asymmetric identity of Rogers--Ramanujan type:
\begin{align*}
\sum_{m,n\ge 0}\frac{(-1)^n q^{2\binom{m}{2}+9\binom{n}{2}+3mn+m+6n}}{(q;q)_m (q^3;q^3)_n} = \frac{1}{(q;q^3)_\infty}.
\end{align*}
They also discovered a similar conjectural identity.
\begin{conjecture}[{The Andrews--Uncu Conjecture \cite[Conjecture 1.2]{AU2021}}]
	\begin{align*}
	\sum_{m,n\ge 0}\frac{(-1)^n q^{2\binom{m}{2}+9\binom{n}{2}+3mn+2m+7n}}{(q;q)_m (q^3;q^3)_n} = \frac{1}{(q^2,q^3;q^6)_\infty}.
	\end{align*}
\end{conjecture}

There are two objects in this paper. First, we establish the following unexpected single-asymmetric Rogers--Ramanujan type identity.

\begin{theorem}\label{th:asy-RR-new}
	\begin{align}\label{eq:asy-RR-new}
	\sum_{n\ge 0} \frac{(-1)^n q^{3\binom{n}{2}+4n}(q;q^3)_n}{(q^9;q^9)_n} = \frac{(q^{4};q^{6})_\infty (q^{12};q^{18})_\infty}{(q^{5};q^{6})_\infty (q^{9};q^{18})_\infty}.
	\end{align}
\end{theorem}

In fact, this identity is the $(a,q)\mapsto (q,q^3)$ case of its $a$-generalization.

\begin{theorem}\label{th:RR-a-generalization}
	\begin{equation}\label{eq:RR-a-generalization}
	\sum_{n\ge 0}\frac{(a;q)_n (a^{-1}q^2;q^2)_n}{(a^2 q;q^2)_n (q^3;q^3)_n}(-1)^n a^n q^{\binom{n}{2}+n} = \frac{(aq;q^2)_\infty (a^3q^3; q^6)_\infty}{(a^2 q; q^2)_\infty (q^3; q^6)_\infty}.
	\end{equation}
\end{theorem}

Using \eqref{eq:asy-RR-new} as a key ingredient, the second part of this paper is devoted to a proof of the Andrews--Uncu Conjecture.

\begin{theorem}
	The Andrews--Uncu Conjecture is true. That is,
	\begin{align}
	\sum_{m,n\ge 0}\frac{(-1)^n q^{2\binom{m}{2}+9\binom{n}{2}+3mn+2m+7n}}{(q;q)_m (q^3;q^3)_n} = \frac{1}{(q^2,q^3;q^6)_\infty}.
	\end{align}
\end{theorem}

\section{The $a$-generalization}

In this section, we prove the $a$-generalization of \eqref{eq:asy-RR-new} given in Theorem \ref{th:RR-a-generalization}. We first observe that, as functions in $a$, both sides of \eqref{eq:RR-a-generalization} are analytic on the set
\begin{align*}
\mathbb{C}\backslash\{\pm q^{-\frac{1}{2}},\pm q^{-\frac{3}{2}},\pm q^{-\frac{5}{2}},\ldots\}.
\end{align*}
Here we shall notice that for nonnegative integers $n$,
\begin{align*}
(a^{-1}q^2;q^2)_n a^n = \prod_{k=1}^n (a - q^{2k}).
\end{align*}
We start by showing that \eqref{eq:RR-a-generalization} is true for $a$ on the open disk $\mathscr{D}:=\{a\in\mathbb{C}:|a|<|q|^{-\frac{1}{2}}\}$. Now, we find that the set $\{q^{2M}:M\in\mathbb{Z}_{\ge 0}\}$ has an accumulation point $0$ which is on the disc $\mathscr{D}$. Therefore, we know from the identity theorem for holomorphic functions that it suffices to prove \eqref{eq:RR-a-generalization} for $a=q^{2M}$ with $M$ nonnegative integers. Let
\begin{align}
S_M := \sum_{n\ge 0}\frac{(q^{2M};q)_n (q^{2-2M};q^2)_n}{(q^{4M+1};q^2)_n (q^3;q^3)_n}(-1)^n q^{\binom{n}{2}+n+2Mn}.
\end{align}
We remark that these summations are terminating.

Our next task is to derive a recurrence for $S_M$. This can be done automatically by the \textit{Mathematica} package \texttt{qZeil} implemented by P. Paule and A. Riese \cite{PR1997}. The package can be downloaded at the website of the Research Institute for Symbolic Computation (RISC) of Johannes Kepler University:
{\small
	\begin{center}
		\url{https://www3.risc.jku.at/research/combinat/software/ergosum/index.html}
	\end{center}
}

To begin with, we import this package.
\begin{lstlisting}[language=Mathematica]
<< RISC`qZeil`
\end{lstlisting}
The recurrence for $S_M$ can be computed by the following codes:
\begin{lstlisting}[language=Mathematica]
ClearAll[n, M];
summand = (
qPochhammer[q^(2 M), q, n] qPochhammer[q^(2 - 2 M), q^2, n])/(
qPochhammer[q^(4 M + 1), q^2, n] qPochhammer[q^3, q^3, n]) (-1)^
n q^(2 M*n) q^(Binomial[n, 2] + n);
qZeil[summand, {n, 0, Infinity}, M, 1]
\end{lstlisting}
The output gives us
\begin{align*}
\mathrm{SUM}[M] == \frac{q^2(1-q^{-3+4M})(1-q^{-1+4M})\,\mathrm{SUM}[-1+M]}{(1-q^{-1+2M})^2 (q^2+q^{4M}+q^{1+2M})}
\end{align*}

Thus, for $M\ge 1$,
\begin{align*}
\frac{S_M}{S_{M-1}} &= \frac{q^2(1-q^{-3+4M})(1-q^{-1+4M})}{(1-q^{-1+2M})^2 (q^2+q^{4M}+q^{1+2M})}\\
&=\frac{(1-q^{4M-3})(1-q^{4M-1})}{(1-q^{2M-1})(1-q^{6M-3})}.
\end{align*}
Recalling that $S_0=1$, we have
\begin{align*}
S_M&=\frac{(q;q^2)_{2M}}{(q;q^2)_M (q^3;q^6)_M}\\
&=\frac{(q^{2M+1};q^2)_\infty (q^{6M+3};q^6)_\infty}{(q^{4M+1};q^2)_\infty (q^3;q^6)_\infty}.
\end{align*}
This is the right-hand side of \eqref{eq:RR-a-generalization} with $a=q^{2M}$.

Now we have proved \eqref{eq:RR-a-generalization} for $a\in\{q^{2M}:M\in\mathbb{Z}_{\ge 0}\}$, and thus for $a\in \mathscr{D}=\{a\in\mathbb{C}:|a|<|q|^{-\frac{1}{2}}\}$ by the previous argument. Finally, the principle of analytic continuation allows us to extend \eqref{eq:RR-a-generalization} to the set $\mathbb{C}\backslash\{\pm q^{-\frac{1}{2}},\pm q^{-\frac{3}{2}},\pm q^{-\frac{5}{2}},\ldots\}$, and finishes the proof.

\section{The Andrews--Uncu Conjecture}

\subsection{A contour integral representation}

Throughout, all contour integrals are over a positively oriented contour separating $0$ from all poles of the integrand.

We recall a general basic contour integral formula given in \cite[p.~127, (4.10.9)]{GR2004} with $m=0$.

\begin{align}\label{eq:general-integral}
&\oint \frac{(a_1 z,\ldots, a_M z,b_1/z,\ldots b_L/z;q)_\infty}{(c_1 z,\ldots c_N z,d_1/z,\ldots, d_L/z;q)_\infty}\frac{\ddd z}{2\pi \iii z}\notag\\
&\quad=\frac{(b_1 c_1,\ldots, b_L c_1,a_1/c_1, \ldots, a_M/c_1;q)_\infty}{(q,d_1 c_1,\ldots, d_L c_1, c_2/c_1,\ldots, c_N/c_1;q)_\infty}\notag\\
&\quad\quad\times {}_{M+L}\phi_{N+L-1}\left(\begin{matrix}
d_1 c_1, \ldots, d_L c_1, q c_1/a_1, \ldots, q c_1/a_M\\
b_1 c_1, \ldots, b_L c_1, q c_1/c_2, \ldots, q c_1/c_N
\end{matrix};q,u(qc_1)^{N-M}\right)\notag\\
&\quad\quad+\operatorname{idem}(c_1;c_2,\ldots,c_N),
\end{align}
where $u=a_1\cdots a_M/(c_1\cdots c_N)$ and $\operatorname{idem}(c_1;c_2,\ldots,c_N)$ stands for the sum of the $N-1$ expressions obtained from the preceding expression by interchanging $c_1$ with each $c_k$ ($k=2,\ldots,N$).

Notice that in \eqref{eq:general-integral}, we shall assume that $M<N$, or $M=N$ with an additional constraint that $|u|<1$.

Let $H$ be the huffing operator, given by
\begin{align}\label{eq:huffing}
H\left (\sum_{n} a(n)q^n\right ):=\sum_{n} a(3n)q^{3n}.
\end{align}

\begin{theorem}\label{th:H-connection}
	Let $\omega=\eee^{2\pi i/3}$. Then
	\begin{align}\label{eq:H-connection}
	&\oint \frac{(q^6 z, q^3 z, 1/z;q^3)_\infty}{(q^4 z, \omega q^4 z, \omega^2 q^4z; q^3)_\infty}\frac{\ddd z}{2\pi \iii z}\notag\\
	&\qquad = H\left(\frac{(q^4,q^2,q^{-1};q^3)_\infty}{(q^9;q^9)_\infty}\sum_{n\ge 0} \frac{(-1)^n q^{3\binom{n}{2}+4n}(q;q^3)_n}{(q^9;q^9)_n}\right).
	\end{align}
\end{theorem}

\begin{proof}
	In \eqref{eq:general-integral}, we replace $q\mapsto q^3$ and choose $M=2$, $N=3$ and $L=1$ with
	\begin{align*}
	(a_1,a_2)&\mapsto (q^6, q^3),\\
	(b_1)&\mapsto (1),\\
	(c_1,c_2,c_3)&\mapsto (q^4, \omega q^4, \omega^2 q^4),\\
	(d_1)&\mapsto (0).
	\end{align*}
	Then
	\begin{align*}
	&\oint \frac{(q^6 z, q^3 z, 1/z;q^3)_\infty}{(q^4 z, \omega q^4 z, \omega^2 q^4z; q^3)_\infty}\frac{\ddd z}{2\pi \iii z}\\
	&\quad = \frac{(q^4,q^2,q^{-1};q^3)_\infty}{(q^3,\omega,\omega^2;q^3)_\infty} {}_{2}\phi_{2}\left(\begin{matrix}
	0,q\\
	\omega q^3, \omega^2 q^3
	\end{matrix};q^3,q^4\right)\\
	&\quad\quad+ \frac{(\omega q^4,\omega^2 q^2,\omega^2 q^{-1};q^3)_\infty}{(q^3,\omega,\omega^2;q^3)_\infty} {}_{2}\phi_{2}\left(\begin{matrix}
	0,\omega q\\
	\omega q^3, \omega^2 q^3
	\end{matrix};q^3,\omega q^4\right)\\
	&\quad\quad+ \frac{(\omega^2 q^4,\omega q^2,\omega q^{-1};q^3)_\infty}{(q^3,\omega,\omega^2;q^3)_\infty} {}_{2}\phi_{2}\left(\begin{matrix}
	0,\omega^2 q\\
	\omega q^3, \omega^2 q^3
	\end{matrix};q^3,\omega^2 q^4\right).
	\end{align*}
	
	Now, let
	\begin{align}\label{eq:F(q)-def}
	F(q)&:=\frac{(q^4,q^2,q^{-1};q^3)_\infty}{(q^3,\omega,\omega^2;q^3)_\infty} {}_{2}\phi_{2}\left(\begin{matrix}
	0,q\\
	\omega q^3, \omega^2 q^3
	\end{matrix};q^3,q^4\right)\notag\\
	&\,=\frac{1}{3}\frac{(q^4,q^2,q^{-1};q^3)_\infty}{(q^9;q^9)_\infty}\sum_{n\ge 0} \frac{(-1)^n q^{3\binom{n}{2}+4n}(q;q^3)_n}{(q^9;q^9)_n}.
	\end{align}
	We observe that
	\begin{align*}
	\frac{(\omega q^4,\omega^2 q^2,\omega^2 q^{-1};q^3)_\infty}{(q^3,\omega,\omega^2;q^3)_\infty} {}_{2}\phi_{2}\left(\begin{matrix}
	0,\omega q\\
	\omega q^3, \omega^2 q^3
	\end{matrix};q^3,\omega q^4\right) = F(\omega q)
	\end{align*}
	and
	\begin{align*}
	\frac{(\omega^2 q^4,\omega q^2,\omega q^{-1};q^3)_\infty}{(q^3,\omega,\omega^2;q^3)_\infty} {}_{2}\phi_{2}\left(\begin{matrix}
	0,\omega^2 q\\
	\omega q^3, \omega^2 q^3
	\end{matrix};q^3,\omega^2 q^4\right) = F(\omega^2 q).
	\end{align*}
	Thus,
	\begin{align*}
	\oint \frac{(q^6 z, q^3 z, 1/z;q^3)_\infty}{(q^4 z, \omega q^4 z, \omega^2 q^4z; q^3)_\infty}\frac{\ddd z}{2\pi \iii z} &= F(q)+F(\omega q)+F(\omega^2 q)\\
	&= 3H\big(F(q)\big).
	\end{align*}
	This gives our desired result by recalling \eqref{eq:F(q)-def}.
\end{proof}

\subsection{$3$-Dissections}

By \eqref{eq:asy-RR-new}, we have
\begin{align}\label{eq:long-series}
&\frac{(q^4,q^2,q^{-1};q^3)_\infty}{(q^9;q^9)_\infty}\sum_{n\ge 0} \frac{(-1)^n q^{3\binom{n}{2}+4n}(q;q^3)_n}{(q^9;q^9)_n}\notag\\
&\qquad = \frac{(q^4,q^2,q^{-1};q^3)_\infty}{(q^9;q^9)_\infty} \frac{(q^{4};q^{6})_\infty (q^{12};q^{18})_\infty}{(q^{5};q^{6})_\infty (q^{9};q^{18})_\infty}\notag\\
&\qquad = (q^{-1},q^2,q^2,q^4,q^4,q^7;q^6)_\infty \frac{(q^{12};q^{18})_\infty}{(q^9;q^9)_\infty (q^9;q^{18})_\infty}\notag\\
&\qquad = (1-q^{-1})(q^5,q^7;q^6)_\infty\frac{(q^2;q^2)_\infty^2}{(q^6;q^6)_\infty^2} \frac{(q^{12};q^{18})_\infty}{(q^9;q^9)_\infty (q^9;q^{18})_\infty}\notag\\
&\qquad = -q^{-1}(q,q^5;q^6)_\infty (q^2;q^2)_\infty^2 \frac{(q^{12};q^{18})_\infty}{(q^6;q^6)_\infty^2 (q^9;q^9)_\infty (q^9;q^{18})_\infty}\notag\\
&\qquad = -q^{-1}(q;q)_\infty (q^2;q^2)_\infty \frac{(q^{12};q^{18})_\infty}{(q^3;q^3)_\infty (q^6;q^6)_\infty (q^9;q^9)_\infty (q^9;q^{18})_\infty}.
\end{align}

Let us recall Ramanujan's classical theta functions
\begin{align*}
\phi(-q)&:=\frac{(q;q)_\infty^2}{(q^2;q^2)_\infty},\\
\psi(q)&:=\frac{(q^2;q^2)_\infty^2}{(q;q)_\infty}.
\end{align*}
Their $3$-dissections are well-known.

\begin{lemma}
	We have
	\begin{align}
	\phi(-q)&= \frac{(q^9;q^9)_\infty^2}{(q^{18};q^{18})_\infty} - 2q \frac{(q^3;q^3)_\infty (q^{18};q^{18})_\infty^2}{(q^6;q^6)_\infty (q^9;q^9)_\infty},\\
	\psi(q)&=\frac{(q^6;q^6)_\infty (q^{9};q^{9})_\infty^2}{(q^3;q^3)_\infty (q^{18};q^{18})_\infty} + q\frac{(q^{18};q^{18})_\infty^2}{(q^9;q^9)_\infty}.
	\end{align}
\end{lemma}

\begin{proof}
	See \cite[(14.3.2) and (14.3.3)]{Hir2017}.
\end{proof}

\begin{theorem}\label{th:H-calculation}
	We have
	\begin{align}\label{eq:H-calculation}
	H\left(\frac{(q^4,q^2,q^{-1};q^3)_\infty}{(q^9;q^9)_\infty}\sum_{n\ge 0} \frac{(-1)^n q^{3\binom{n}{2}+4n}(q;q^3)_n}{(q^9;q^9)_n}\right) = \frac{1}{(q^3;q^3)_\infty (q^6,q^9;q^{18})_\infty}.
	\end{align}
\end{theorem}

\begin{proof}
	We conclude from \eqref{eq:long-series} that
	\begin{align*}
	&H\left(\frac{(q^4,q^2,q^{-1};q^3)_\infty}{(q^9;q^9)_\infty}\sum_{n\ge 0} \frac{(-1)^n q^{3\binom{n}{2}+4n}(q;q^3)_n}{(q^9;q^9)_n}\right)\\
	&\qquad=-\frac{(q^{12};q^{18})_\infty}{(q^3;q^3)_\infty (q^6;q^6)_\infty (q^9;q^9)_\infty (q^9;q^{18})_\infty}H\big(q^{-1}(q;q)_\infty (q^2;q^2)_\infty\big)\\
	&\qquad = -\frac{(q^{12};q^{18})_\infty}{(q^3;q^3)_\infty (q^6;q^6)_\infty (q^9;q^9)_\infty (q^9;q^{18})_\infty}H\big(q^{-1}\phi(-q)\psi(q)\big)\\
	&\qquad= \frac{(q^{12};q^{18})_\infty}{(q^3;q^3)_\infty (q^6;q^6)_\infty (q^9;q^9)_\infty (q^9;q^{18})_\infty} \cdot (q^9;q^9)_\infty (q^{18};q^{18})_\infty.
	\end{align*}
	This simplifies to our desired result.
\end{proof}

\subsection{Proof of the Andrews--Uncu Conjecture}

Now we complete the proof of the Andrews--Uncu Conjecture. Notice that
\begin{align*}
&\sum_{m,n\ge 0}\frac{(-1)^n q^{2\binom{m}{2}+9\binom{n}{2}+3mn+2m+7n}}{(q;q)_m (q^3;q^3)_n}\\
&\qquad=\oint \sum_{m\ge 0}\frac{(-1)^m q^{\binom{m}{2}+2m} z^m}{(q;q)_m}\sum_{n\ge 0} \frac{q^{4n} z^{3n}}{(q^3;q^3)_n} \sum_{\ell=-\infty}^\infty (-1)^\ell q^{\binom{\ell}{2}} z^{-\ell} \frac{\ddd z}{2\pi \iii z}\\
&\qquad= \oint \frac{(q^2 z;q)_\infty (q,qz,1/z;q)_\infty}{(q^4z^3;q^3)_\infty} \frac{\ddd z}{2\pi \iii z}\\
&\qquad= (q;q)_\infty \oint \frac{(q^2 z, q z, 1/z;q)_\infty}{(q^{4/3} z, \omega q^{4/3} z, \omega^2 q^{4/3} z; q)_\infty}\frac{\ddd z}{2\pi \iii z},
\end{align*}
where we still put $\omega=\eee^{2\pi i/3}$. We further observe that the above contour integral is indeed the integral in \eqref{eq:H-connection} with $q$ replaced by $q^{1/3}$. By \eqref{eq:H-connection} and \eqref{eq:H-calculation}, we conclude that
\begin{align*}
\sum_{m,n\ge 0}\frac{(-1)^n q^{2\binom{m}{2}+9\binom{n}{2}+3mn+2m+7n}}{(q;q)_m (q^3;q^3)_n}&=(q;q)_\infty \cdot \frac{1}{(q;q)_\infty (q^2,q^3;q^{6})_\infty}\\
&=\frac{1}{(q^2,q^3;q^{6})_\infty}.
\end{align*}

\section{Conclusion}

The entire surprise comes from the single-asymmetric Rogers--Ramanujan type identity \eqref{eq:asy-RR-new}, which is \textit{completely weird}, as pointed out by George Andrews in a personal communication. Although its $a$-generalization \eqref{eq:RR-a-generalization} presents a similar shape to the Lebesgue identity \eqref{eq:Leb} or the Corteel--Savage--Sills identity \eqref{eq:CSS}, the unexpectedness rises as two extra $q$-shifted factorials are attached to the summand in \eqref{eq:RR-a-generalization} in comparison to the original \eqref{eq:asy-RR-new}. It is also natural to ask if there are other instances that behave as ``unreasonable'' as \eqref{eq:asy-RR-new}. A thorough computer search will be carried out in a forthcoming project.

In addition, \eqref{eq:asy-RR-new} gives a $q$-analog of the following identity:
\begin{align*}
&\sum_{n\ge 0}\frac{(1/3)_n}{n!}\left({-\frac{1}{3}}\right)^n\\
&\qquad= {}_{1}F_{0}\left(\begin{matrix}
1/3\\
-
\end{matrix};-\frac{1}{3}\right)= \left(\frac{3}{4}\right)^{\frac{1}{3}} \qquad\text{(by \cite[p.~64, (2.1.6)]{AAR1999})}\\
&\qquad= \frac{\Gamma(\frac{5}{18})\Gamma(\frac{11}{18})\Gamma(\frac{17}{18})\Gamma(\frac{1}{2})}{\Gamma(\frac{2}{9})\Gamma(\frac{5}{9})\Gamma(\frac{8}{9})\Gamma(\frac{2}{3})}\notag\\
&\qquad= \prod_{n\ge 0}\frac{(n+\frac{2}{9})(n+\frac{5}{9})(n+\frac{8}{9})(n+\frac{2}{3})}{(n+\frac{5}{18})(n+\frac{11}{18})(n+\frac{17}{18})(n+\frac{1}{2})}. \qquad\text{(by \cite[p.~247]{WW2021})}
\end{align*}
It is remarkable that in this limiting case (for $q\to 1^{-}$), the ${}_{2}\phi_{2}$ (or ${}_{3}\phi_{2}$) series in \eqref{eq:asy-RR-new} reduces to a ${}_{1}F_{0}$ series of argument $-\frac{1}{3}$.

\subsection*{Acknowledgements}

This work was supported by a Killam Postdoctoral Fellowship from the Killam Trusts.

\bibliographystyle{amsplain}

\end{document}